\documentclass[12pt]{amsart}
\usepackage{amssymb,latexsym}
\usepackage{pdfsync}

\newdimen\AAdi%
\newbox\AAbo%
\def\AAk#1#2{\s_etbox\AAbo=\hbox{#2}\AAdi=\wd\AAbo\kern#1\AAdi{}}%
\def\AAr#1#2#3{\s_etbox\AAbo=\hbox{#2}\AAdi=\ht\AAbo\raise#1\AAdi\hbox{#3}}%
\font\tenmsb=msbm10 at 12pt
\font\sevenmsb=msbm7 at 8pt
\font\fivemsb=msbm5 at 6pt
\newfam\msbfam
\textfont\msbfam=\tenmsb
\scriptfont\msbfam=\sevenmsb
\scriptscriptfont\msbfam=\fivemsb
\def\Bbb#1{{\tenmsb\fam\msbfam#1}}

\newcommand{\beq}{\begin{equation}}
\newcommand{\eeq}{\end{equation}}
\newcommand{\beqr}{\begin{eqnarray}}
\newcommand{\eeqr}{\end{eqnarray}}
\newcommand{\ba}{\begin{array}}
\newcommand{\ea}{\end{array}}

\begin{document}

\newtheorem{thm}{Theorem}
\newtheorem{lem}{Lemma}
\newtheorem{cor}{Corollary}
\newtheorem{rem}{Remark}
\newtheorem{pro}{Proposition}
\newtheorem{defi}{Definition}
\newtheorem{eg}{Example}
\newtheorem*{claim}{Claim}
\newtheorem{conj}[thm]{Conjecture}
\newcommand{\noi}{\noindent}
\newcommand{\dis}{\displaystyle}
\newcommand{\mint}{-\!\!\!\!\!\!\int}
\numberwithin{equation}{section}

\def \bx{\hspace{2.5mm}\rule{2.5mm}{2.5mm}}
\def \vs{\vspace*{0.2cm}}
\def\hs{\hspace*{0.6cm}}
\def \ds{\displaystyle}
\def \p{\partial}
\def \O{\Omega}
\def \o{\omega}
\def \b{\beta}
\def \m{\mu}
\def \l{\lambda}
\def\L{\Lambda}
\def \ul{u_\lambda}
\def \D{\Delta}
\def \d{\delta}
\def \k{\kappa}
\def \s{\sigma}
\def \e{\varepsilon}
\def \a{\alpha}
\def \tf{\tilde{f}}
\def\cqfd{%
\mbox{ }%
\nolinebreak%
\hfill%
\rule{2mm} {2mm}%
\medbreak%
\par%
}
\def \pr {\noindent {\it Proof.} }
\def \rmk {\noindent {\it Remark} }
\def \esp {\hspace{4mm}}
\def \dsp {\hspace{2mm}}
\def \ssp {\hspace{1mm}}

\def\la{\langle}\def\ra{\rangle}

\def \u{u_+^{p^*}}
\def \ui{(u_+)^{p^*+1}}
\def \ul{(u^k)_+^{p^*}}
\def \energy{\int_{\R^n}\u }
\def \sk{\s_k}
\def \mo{\mu_k}
\def\cal{\mathcal}
\def \I{{\cal I}}
\def \J{{\cal J}}
\def \K{{\cal K}}
\def \OM{\overline{M}}

\def\n{\nabla}

\def\fk{{{\cal F}}_k}
\def\M1{{{\cal M}}_1}
\def\Fk{{\cal F}_k}
\def\Fl{{\cal F}_l}
\def\FF{\cal F}
\def\Gk{{\Gamma_k^+}}
\def\n{\nabla}
\def\uuu{{\n ^2 u+du\otimes du-\frac {|\n u|^2} 2 g_0+S_{g_0}}}
\def\uuug{{\n ^2 u+du\otimes du-\frac {|\n u|^2} 2 g+S_{g}}}
\def\sku{\sk\left(\uuu\right)}
\def\qed{\cqfd}
\def\vvv{{\frac{\n ^2 v} v -\frac {|\n v|^2} {2v^2} g_0+S_{g_0}}}
\def\vvs{{\frac{\n ^2 \tilde v} {\tilde v}
 -\frac {|\n \tilde v|^2} {2\tilde v^2} g_{S^n}+S_{g_{S^n}}}}
\def\skv{\sk\left(\vvv\right)}
\def\tr{\hbox{tr}}
\def\pO{\partial \Omega}
\def\dist{\hbox{dist}}
\def\RR{\Bbb R}\def\R{\Bbb R}
\def\C{\Bbb C}
\def\B{\Bbb B}
\def\N{\Bbb N}
\def\Q{\Bbb Q}
\def\Z{\Bbb Z}
\def\PP{\Bbb P}
\def\EE{\Bbb E}
\def\F{\Bbb F}
\def\G{\Bbb G}
\def\H{\Bbb H}
\def\SS{\Bbb S}\def\S{\Bbb S}

\def\div{\hbox{div}\,}

\def\lcf{{locally conformally flat} }

\def\circledwedge{\setbox0=\hbox{$\bigcirc$}\relax \mathbin {\hbox
to0pt{\raise.5pt\hbox to\wd0{\hfil $\wedge$\hfil}\hss}\box0 }}

\def\sss{\frac{\s_2}{\s_1}}

\date{\today}
\title[ Rigidity of symplectic translating solitons ]{ Rigidity of symplectic translating solitons }

\author{}

\author[Qiu]{Hongbing Qiu}
\address{School of Mathematics and Statistics\\ Wuhan University\\Wuhan 430072,
China,
and Hubei Key Laboratory of Computational Science \\ Wuhan University\\Wuhan 430072,
China
 }
 \email{hbqiu@whu.edu.cn}
 
  \thanks{
 This work is partially supported by NSFC (No. 11771339) and Hubei Provincial Natural Science Foundation of China (No. 2021CFB400).  }

\begin{abstract}

We obtain a rigidity result of symplectic translating solitons via the complex phase map. It indicates that we can remove the bounded second fundamental form assumption for symplectic translating solitons in \cite{HanSun10}.

\vskip12pt



\end{abstract}
\maketitle
\section{Introduction}

Let $X: M^{n} \rightarrow \mathbb{R}^{m+n}$ be an isometric immersion from an $n$-dimensional oriented Riemannian manifold $M$ to the Euclidean space $\mathbb{R}^{m+n}$.  
The mean curvature flow (MCF) in Euclidean space is a one-parameter family of immersions $X_t= X(\cdot, t): M^n \rightarrow \mathbb{R}^{m+n}$ with the corresponding image $M_t=X_t(M)$ such that
\begin{equation}\label{ODE-add}
\begin{cases}\aligned
\frac{\partial}{\partial t}X(x,t)=& H(x,t), x\in M,\\
X(x,0)=&X(x),
\endaligned
\end{cases}
\end{equation}
is satisfied, where $H(x, t)$ is the mean curvature vector of $M_t$ at $X(x, t)$ in $\mathbb{R}^{m+n}$.

$M^n$ is said to be a translating soliton in $\mathbb{R}^{m+n}$ if it satisfies
\begin{equation}\label{eqn-T111}
H=-V_{0}^N,
\end{equation}
where $V_0$ is a fixed vector in $\mathbb{R}^{m+n}$ with unit length and $V_{0}^N$ denotes the orthogonal projection of
$V_0$ onto the normal bundle of $M^n$.  
The translating soliton plays an important role in analysis of singularities in mean curvature. It is not only a special solution to the mean curvature flow, but also it is one of the most important example of Type $\rm II$ singularity (see \cite{AV95, AV97, Ham95, HuiSin99, HuiSin2-99, Whi00, Whi03}). And the geometry of the translating soliton has been paid considerable attention during the past two decades, see the references (not exhaustive): \cite{CheQiu16, CSS, Hal, Jian, Kun, MSS, Ngu09, Ngu13, NT, Wang11, Xin}, etc.

Since the translating soliton can be regarded as a generalization of minimal submanifolds, it is natural to study the Bernstein type problem of translating solitons. For the codimension one case, Bao-Shi \cite{BS} showed a translating soliton version of the Moser's theorem for minimal hypersurfaces (\cite{Mos}), namely, if the image of the translating soliton $M^n$ under the Gauss map is contained in a regular ball, then such a complete translating soliton in $\mathbb{R}^{n+1}$ has to be a hyperplane. For higher codimensions, Kunikawa \cite{Kun} generalized the result of \cite{BS} to the flat normal bundle case. 
 In general, Xin \cite{Xin} proved that if the $v$-function satisfies $v\leq v_1 < v_0:= \frac{2\cdot 3^{\frac{2}{3}}}{1+3^{\frac{2}{3}}}$, then any complete translating soliton $M^n$ in $\mathbb{R}^{m+n}_n (m\geq 2)$ must be affine linear. Recently, this result was improved by using a new test function in \cite{Qiu}.
 
 One special class of translating solitons with higher codimension is the symplectic translating solitons, which are solutions to symplectic MCFs (see \cite{CheLi01, Wan01}). By using the Omori-Yau maximum principle, Han-Li \cite{HanLi09} gave an estimate of the k\"ahler angle for symplectic standard translating solitons, where the second fundamental form $B$ satisfies $\max |B| = 1$. Afterward, Han-Sun \cite{HanSun10} showed that if the cosine of the K\"ahler angle has a positive lower bound, then any complete symplectic translating soliton with nonpositive normal curvature and bounded second fundamental form has to be an affine plane. Inspired by the work of Colding-Minicozzi \cite{ColMin12} for self-shrinkers, it is natural to ask that whether we can remove the bounded second fundamental form condition for symplectic translating solitons in the previous result of \cite{HanSun10}.
 
 Recall that Leung-Wan \cite{LeuWan07}  introduced the concept of hyper-Lagrangian manifolds which is a generalization of complex Lagrangian submanifolds: A submanifold $L^{2n}$ of a hperk\"ahler manifold $ {\widetilde M}^{4n}$ is called \emph{hyper-Lagrangian} if each tangent space $T_xL \subset T_x {\widetilde M}$ is a complex Lagrangian subspace with respect to  $\Omega_{J(x)}$ with varying $J(x) \in \mathbb{S}^2$. Such a map $x \rightarrow J(x)$ is called the \emph{complex phase map} $J: L \rightarrow \mathbb{S}^2$. We observe that any  oriented surface immersed in a hyperk\"ahler $4$-manifold is always hyper-Lagrangian, and in a hyperk\"ahler 4-manifold, a surface being symplectic is equivalent to the condition that the image under the complex phase map is contained in an open hemisphere. 
 
 In this note, firstly, by applying the fact that there is always a hyper-Lagrangian structure on a 2-dimensional translating soliton in $\mathbb{R}^4$, we show that the complex phase map of a 2-dimensional translating soliton is a generalized harmonic map, which can be regarded as a counterpart of the Ruh-Vilms type result for 2-dimensional translating solitons (Theorem \ref{thm:translator}). Secondly, by using Theorem \ref{thm:translator} and a test function in terms of the mean curvature and the complex phase map, we prove that if the image of the complex phase map is contained in a regular ball in $\mathbb{S}^2$, then any complete symplectic translating soliton with nonpositive normal curvature has to be an affine plane (Theorem \ref{thm:soliton}). Note that the image of the complex phase map being contained in a regular ball is equivalent to the condition that the cosine of the K\"ahler angle has a positive lower bound. Therefore Theorem \ref{thm:soliton} implies that the rigidity result for symplectic translating solitons in \cite{HanSun10} holds even without the bounded second fundamental form assumption.

\vskip24pt

\section{Preliminaries}

In this section, we give some notations that will be used throughout the paper and recall some relevant definitions and results.

Let ${\widetilde M}^{4n}$ be a $4n$-dimensional  hyperk\"ahler manifold, i.e., there exists two covariant constant anti-commutative almost complex structures $J_1, J_2$, i.e., $J_1, J_2$ are parallel with respect to the Levi-Civita connection and $J_1J_2=-J_2J_1$. Denote $J_3:= J_1J_2$, then the following quaternionic identities hold
\begin{equation*}
J_1^2=J_2^2=J_3^2=J_1J_2J_3=-1. 
\end{equation*}

Every $\mathrm{SO}(3)$ matrix preserves the quaternionic identities, i.e., $\{\widetilde J_{\alpha}:=\sum_{\beta=1}^3a_{\alpha\beta}J_{\beta}\}$ satisfies the quaternionic identities
\begin{equation*}
\widetilde J_1^2=\widetilde J_2^2=\tilde J_3^2=\widetilde J_1\widetilde J_2\widetilde J_3=-1.
\end{equation*}
In particular, for every unit vector $(a_1,a_2,a_3)\in\mathbb{R}^3$, we get a covariant constant almost complex structure $\sum_{\alpha=1}^3a_{\alpha}J_{\alpha}$, and this implies that $\left(M,\sum_{\alpha=1}^3a_{\alpha}J_{\alpha}\right)$ is a K\"ahler manifold.

Let $\widehat J=\sum_{\alpha=1}^3\lambda_{\alpha}J_{\alpha}$ be an almost complex structure on ${\widetilde M}$. Let $\omega_{\widehat J}$ be the K\"ahler form with respect to $\widehat J$, then the associated symplectic $2$-form $\Omega_{\widehat J}\in\Omega^{2,0}\left({\widetilde M},\widehat J\right)$ is given by
\begin{equation*}
\Omega_{\widehat J}=\omega_{K}+\sqrt{-1}\omega_{K\widehat J},
\end{equation*}
where $K=\sum_{\alpha=1}^3\mu_{\alpha}J_{\alpha}$ is an almost complex structure which is orthogonal to $\widehat J$ in the sense that $\sum_{\alpha=1}^3\lambda_{\alpha}\mu_{\alpha}=0$. If $\widehat J$ is parallel, then $\Omega_{\widehat J}$ is holomorphic with respect to the covraiant constant almost complex structure $\widehat J$.

Let $\omega_{\alpha}$ be the K\"ahler form associated with the almost complex structure $J_{\alpha}$, then $\left({\widetilde M},J_1\right)$ is a K\"ahler manifold and 
\begin{align*}
\Omega_{J_1}=\omega_2+\sqrt{-1}\omega_3\in H^{2,0}\left(M,J_1\right)
\end{align*}
is the associated holomorphic symplectic $2$-form. We say that a $2n$-dimensional submanifold $L^{2n}$ of ${\widetilde M}^{4n}$ is complex Lagrangian if for some covariant constant complex structure $\widehat J$ of ${\widetilde M}$ such that the associated holomorphic symplectic $2$-form $\Omega_{\widehat J}$ vanished everywhere on $L$. Without loss of generality, assume $\widehat J=J_1$, then $L$ is a K\"ahler submanifold of the K\"ahler manifold $({\widetilde M}, J_1)$. In particular, $L$ is a minimal submanifold of ${\widetilde M}$. Moreover, both $L\subset\left({\widetilde M}, J_2\right)$ and $L\subset\left({\widetilde M}, J_3\right)$ are  Lagrangian immersions.

We say that $L^{2n}$ is a hyper-Lagrangian submanifold of ${\widetilde M}^{4n}$ if there is an almost complex structure $\widehat J=\sum_{\alpha=1}^3\lambda_{\alpha}J_{\alpha}$ such that the associated symplectic $2$-form $\Omega_{\widehat J}$ vanished everywhere on $L$.  The map 
\begin{equation*}
J:L\rightarrow\mathbb{S}^2,\quad x\mapsto J(x):=(\lambda_1,\lambda_2,\lambda_3)
\end{equation*}
is called the complex phase map. In other words, $L$ is hyper-Lagrangian iff each $T_xL$ is a  complex Lagrangian subspace of $T_x{\widetilde M}$. Here we say that $T_xL$ is a complex Lagrangian subspace of $T_x{\widetilde M}$ if for some complex structure $\widehat J=\sum_{\alpha=1}^3\lambda_{\alpha}J_{\alpha}$ we have
\begin{equation*}
\bar g\left(K\cdot,\cdot\right)\vert_{T_xL}=0
\end{equation*}
for all almost complex structures $K=\sum_{\alpha=1}^3\mu_{\alpha}J_{\alpha}$ which are orthogonal to $\widehat  J$. Therefore, $L$ is complex Lagrangian iff $L$ is hyper-Lagrangian with the  constant complex phase map.

The complex phase map $J$ defines an almost complex structure $\widetilde J=\sum_{\alpha=1}^3\lambda_{\alpha}J_{\alpha}\vert_{TL\rightarrow TL}$ on $L$ and an almost complex structure $\widetilde J^{\bot}=\sum_{\alpha=1}^3\lambda_{\alpha}J_{\alpha}\vert_{T^{\bot}L\rightarrow T^{\bot}L}$ on $T^{\bot}L$, where $T^{\bot}L$ is the normal bundle of $L$ in ${\widetilde M}$. We denote the  Levi-Civita connection on $T{\widetilde M},TL$ by $\overline\nabla, \nabla$ respectively, if there is no confusion, we also denote the normal connection on $T^{\bot}$L by $ \nabla$. 
For $V\in \Gamma(TL)$, let $\Delta_V:= \Delta + \left< V, \nabla \cdot \right>$, where $\Delta$ is the usual Laplacian operator with respect to $\nabla$ on $L$.

\bigskip

The second fundamental form $B$ of $L^{2n}$ in ${\widetilde M}^{4n}$ is defined by
\[
B(U, W):= (\overline{\n}_U W)^N
\]
for $U, W \in \Gamma(TL)$. We use the notation $( \cdot )^T$ and $(
\cdot )^N$ for the orthogonal  projections into the tangent bundle
$TL$ and the normal bundle $T^{\bot}L$, respectively. For $\nu \in
\Gamma(T^{\bot}L)$ we define the shape operator $A^\nu: TL \rightarrow TL$
by
\[
A^\nu (U):= - (\overline{\n}_U \nu)^T
\]
Taking the trace of $B$ gives the mean curvature vector $H$ of $L$
in ${\widetilde M}^{4n}$ and
\[
H:= \hbox{trace} (B) = \sum_{i=1}^{2n} B(e_i, e_i),
\]
where $\{ e_i \}$ is a local orthonormal frame field of $L$.

The curvature tensors $R(X, Y)Z$ and $R(X, Y)\mu$ can be defined, corresponding to the connections in the tangent bundle and the normal bundle respectively, as follows:
\begin{equation*}\aligned
R(X, Y)Z =& -\n_X\n_Y Z + \n_Y\n_X Z +\n_{[X, Y]}Z, \\
R(X, Y)\mu =& -\n_X\n_Y \mu + \n_Y\n_X \mu +\n_{[X, Y]}\mu,
\endaligned
\end{equation*}
where $X, Y, Z$ are tangent vector fields, $\mu$ is a normal vector field.

\vskip24pt

\section{The complex phase map}

Firstly, we give the following characterization of the hyper-Lagrangian condition:

\begin{lem}(\cite{QS})\label{lem:hyper-Lagrangian} Let $L^{2n}$ be a submanifold of a hyperk\"ahler manifold ${\widetilde M}^{4n}$. Then $L$ is hyper-Lagrangian iff
\begin{equation*}
J_{\alpha}\vert_{TL\rightarrow TL}=\lambda_{\alpha}\widetilde J,\quad\alpha=1,2,3,
\end{equation*}
iff
\begin{equation*}
J_{\alpha}\vert_{T^{\bot}L\rightarrow T^{\bot}L}=\lambda_{\alpha}\widetilde J^{\bot},\quad\alpha=1,2,3.
\end{equation*}
\end{lem}
\begin{proof}
For the completeness, we write the details of the proof as follows:

Under the orthogonal decomposition $T_x{\widetilde M}=T_xL\oplus T_x^{\bot}L$, we write
\begin{equation*}
J_{\alpha}=\begin{pmatrix}A_{\alpha}&B_{\alpha}\\
-B_{\alpha}^T&C_{\alpha}
\end{pmatrix},\quad 
\alpha=1,2,3.
\end{equation*}
Let $A=(a_{\alpha\beta})_{1\leq \alpha,\beta\leq 3}\in\mathrm{SO}(3)$ where $\lambda_{\alpha}=a_{1\alpha}$. Set $\widetilde J_{\beta}=\sum_{\alpha=1}^3a_{\beta\alpha}J_{\alpha}$. Since $L$ is hyper-Lagrangian, we get
\[
\sum_{\gamma=1}^3 a_{\b \gamma}A_\gamma =0, \quad \b=2, 3.
\]
Namely
\begin{equation*}
\widetilde J_{\b}=\begin{pmatrix}0&\sum_{\alpha=1}^3a_{\b\alpha}B_{\alpha}\\
-\sum_{\alpha=1}^3a_{\b\alpha}B_{\alpha}^{T}&\sum_{\alpha=1}^3a_{\b\alpha}C_{\alpha}
\end{pmatrix}, \quad \b=2, 3.
\end{equation*}
Since $\widetilde{J}_{\b}^2 = -1$, we can conclude that 
\[
 \sum_{\gamma=1}^3 a_{\b \gamma}C_\gamma =0, \quad \b=2, 3.
\]
Hence we have
\begin{equation*}
\widetilde J_{2}=\begin{pmatrix}0&\sum_{\alpha=1}^3a_{2\alpha}B_{\alpha}\\
-\sum_{\alpha=1}^3a_{2\alpha}B_{\alpha}^{T}&0
\end{pmatrix},\quad \widetilde J_{3}=\begin{pmatrix}0&\sum_{\alpha=1}^3a_{3\alpha}B_{\alpha}\\
-\sum_{\alpha=1}^3a_{3\alpha}B_{\alpha}^{T}&0
\end{pmatrix},
\end{equation*}
and 
\begin{equation*}
A_{\alpha}=\sum_{\b, \gamma=1}^3 a_{\a\b}a_{\b\gamma}A_{\gamma} = \sum_{\gamma=1}^3 a_{1\a}a_{1\gamma}A_\gamma = a_{1\a} \widetilde{J} = \lambda_{\alpha}\widetilde J,\quad\alpha=1,2,3,
\end{equation*}
\begin{equation*}
C_{\alpha}=\sum_{\b, \gamma=1}^3 a_{\a\b}a_{\b\gamma}C_{\gamma} = \sum_{\gamma=1}^3 a_{1\a}a_{1\gamma}C_\gamma = a_{1\a} \widetilde{J}^{\bot} =  \lambda_{\alpha}\widetilde J^{\bot},\quad\alpha=1,2,3.
\end{equation*}
Conversely, if $A_{\alpha}= \lambda_{\alpha}\widetilde J (\a = 1, 2, 3),$ it follows that 
\[
\sum_{\gamma=1}^3 a_{\b \gamma}A_\gamma = \sum_{\gamma =1}^3 a_{\b\gamma} \l_{\gamma} \widetilde{J} = \sum_{\gamma =1}^3 a_{\b\gamma} a_{1\gamma} \widetilde{J} =\d_{\b 1} \widetilde{J}= 0, \quad \b=2, 3. 
\]
Similarly,  if $C_{\alpha}=   \lambda_{\alpha}\widetilde J^{\bot} (\a= 1, 2, 3)$, 
\[
\sum_{\gamma=1}^3 a_{\b \gamma}C_\gamma = \sum_{\gamma =1}^3 a_{\b\gamma} \l_\gamma \widetilde{J}^{\bot} = \sum_{\gamma =1}^3 a_{\b\gamma} a_{1\gamma} \widetilde{J}^{\bot} =\d_{\b 1} \widetilde{J}^{\bot}= 0, \quad \b=2, 3. 
\]
Then we get
\begin{equation*}
\widetilde J_{\b}=\begin{pmatrix}\sum_{\alpha=1}^3a_{\b\alpha}A_{\alpha}&\sum_{\alpha=1}^3a_{\b\alpha}B_{\alpha}\\
-\sum_{\alpha=1}^3a_{\b\alpha}B_{\alpha}^{T}&0
\end{pmatrix}, \quad \b=2, 3.
\end{equation*}
Since $\widetilde{J}_{\b}^2 = -1$, we can conclude that 
\[
 \sum_{\gamma=1}^3 a_{\b \gamma}A_\gamma =0, \quad \b=2, 3.
\]
 Hence we complete the proof.
\end{proof}
\begin{rem}\label{rem1}
This Lemma claims that any oriented surface immersed in a hyperk\"ahler $4$-manifold is always hyper-Lagrangian.
\end{rem}

 Recall that a map $u$ from a Riemannian manifold $(M,g)$ to another Riemannian manifold $(N,h)$
is called a $V$-harmonic map if it solves
\[\tau (u)+du(V)=0,\]
where $\tau(u)$ is the tension field of the map $u$, and $V$ is a vector field on $M$ (cf.\cite{CheJosWan15, CheJosQiu12}). Clearly, it is a generalization of the usual harmonic map.

According to Remark \ref{rem1}, a 2-dimensional translating soliton $\Sigma$ in $\mathbb{R}^4$ is hyper-Lagrangian. By applying this new structure on it, we show that the complex phase map of $\Sigma$ is a $-V_0^{T}$-harmonic map as follows:.

\begin{thm}\label{thm:translator}Let $X:\Sigma^2\rightarrow\mathbb{R}^4$ be a 2-dimensional translating soliton. Then the complex phase map $J:\Sigma\rightarrow\mathbb{S}^2$ satisfies
\begin{equation*}
\tau\left(J\right)= dJ\left(V_0^{T}\right).
\end{equation*}
\end{thm}

\begin{proof}

Since  $\Sigma$ is hyper-Lagrangian, by  Lemma \ref{lem:hyper-Lagrangian}, for any $Z\in \Gamma(T\Sigma)$, we get 
\begin{equation}\label{eqn-J666}
\l_\a\widetilde{J}Z =(J_\a Z)^T.
\end{equation}
Then we take the covariant differential on both sides of the above equation with any $Y\in \Gamma(T\Sigma)$,
\begin{equation}\label{eqn-cova}
\n_Y (\l_\a\widetilde{J}Z) = \n_Y (J_\a Z)^T.
\end{equation}
Let $\{e_1, e_2\}$ be a local orthonormal frame field of $T\Sigma$, and $\n e_i = 0$ at the considered point. The Gauss formula and (\ref{eqn-J666}) imply that 
\begin{equation}\label{eqn-RHS}\aligned
&\n_Y (J_\a Z)^T = \overline{\n}_Y(J_\a Z)^T - B(Y, (J_\a Z)^T) \\
=& \overline{\n}_Y \left(\sum_{j=1}^2 \la J_\a Z, e_j \ra e_j \right) - B(Y, \l_\a \widetilde{J}Z) \\
=&  \sum_{j=1}^2 \left[ \left( Y\la J_\a Z, e_j \ra \right) e_j + \la J_\a Z, e_j \ra \overline{\n}_Y e_j \right] - B(Y, \l_\a \widetilde{J}Z) \\
=& (\overline{\n}_Y(J_\a Z))^T + \sum_{j=1}^2 \la J_\a Z, B(Y, e_j) \ra e_j + \sum_{j=1}^2 \la J_\a Z, e_j \ra B(Y, e_j) - \l_\a B(Y,  \widetilde{J}Z) \\
\endaligned
\end{equation}
Since
\begin{equation}\label{eqn-J1}\aligned
(\overline{\n}_Y(J_\a Z))^T = & ((\overline{\n}_Y J_\a) Z + J_\a \overline{\n}_Y Z)^T \\
=& (J_\a \overline{\n}_Y Z)^T = (J_\a \n_Y Z + J_\a B(Y, Z))^T,
\endaligned
\end{equation}
and 
\begin{equation}\label{eqn-JB1}\aligned
\sum_{j=1}^2 \la J_\a Z, e_j \ra B(Y, e_j) = & \sum_{j=1}^2 \la (J_\a Z)^T, e_j \ra B(Y, e_j) \\
=& \sum_{j=1}^2 \la \l_\a \widetilde{J}Z, e_j \ra B(Y, e_j) = \l_\a B(Y, \widetilde{J}Z).
\endaligned
\end{equation}
Substituting (\ref{eqn-J1}) and (\ref{eqn-JB1}) into (\ref{eqn-RHS}), we obtain
\begin{equation}\label{eqn-RHS2}\aligned
\n_Y (J_\a Z)^T = &  (J_\a \n_Y Z + J_\a B(Y, Z))^T +  \sum_{j=1}^2 \la (J_\a Z)^N, B(Y, e_j) \ra e_j \\
=&  (J_\a \n_Y Z )^T+ (J_\a B(Y, Z))^T +  \sum_{j=1}^2 \la A^{(J_\a Z)^N} (Y), e_j \ra e_j \\
=& \l_\a \widetilde{J} \n_Y Z + (J_\a B(Y, Z))^T +  A^{(J_\a Z)^N} (Y)
\endaligned
\end{equation}
By (\ref{eqn-cova}) and (\ref{eqn-RHS2}), we have 
\begin{equation*}\label{eqn-RHS3}\aligned
Y(\l_\a)\widetilde{J}Z + \l_\a (\n_Y \widetilde{J})Z +\l_\a \widetilde{J}\n_Y Z = & \n_Y (\l_\a \widetilde{J}Z) \\
= & \n_Y (J_\a Z)^T \\
= &  \l_\a \widetilde{J} \n_Y Z + (J_\a B(Y, Z))^T +  A^{(J_\a Z)^N} (Y).
\endaligned
\end{equation*}
Namely
\begin{equation}\label{eqn-LRH}\aligned
Y(\l_\a)\widetilde{J}Z + \l_\a (\n_Y \widetilde{J})Z  =  (J_\a B(Y, Z))^T +  A^{(J_\a Z)^N} (Y).
\endaligned
\end{equation}
Multiplying $\l_\a$ on both sides of the above equality and take summation from $\a=1$ to $3$, we get 
\begin{equation*}\label{eqn-RHS3}\aligned
Y(\sum_{\a=1}^3 \l_\a^{2})\widetilde{J}Z + \sum_{\a=1}^3\l_\a^{2} (\n_Y \widetilde{J})Z  =  (\sum_{\a=1}^3\l_\a J_\a B(Y, Z))^T +  A^{(\sum_{\a=1}^3\l_\a J_\a Z)^N} (Y).
\endaligned
\end{equation*}
Note that $\sum_{\a=1}^3\l_\a^{2}=1$ and $\sum_{\a=1}^3\l_\a J_\a = \widetilde{J}_1$, where $\widetilde{J}_1$ is the same as the one in the proof of Lemma \ref{lem:hyper-Lagrangian}, we then derive
\begin{equation*}\label{eqn-RHS3}\aligned
(\n_Y \widetilde{J})Z = (\widetilde{J}_1 B(Y, Z))^T + A^{(\widetilde{J}_1Z)^N} (Y).
\endaligned
\end{equation*}
Since $\widetilde{J}_1 = J = \widetilde{J} \oplus \widetilde{J}^{\bot}$, we have
\[
(\widetilde{J}_1Z)^N = 0  \quad  (\widetilde{J}_1 B(Y, Z))^T=0.
\]
Therefore 
\begin{equation}\label{eqn-J}
\n \widetilde{J}=0.
\end{equation}
Combining (\ref{eqn-LRH}) with (\ref{eqn-J}), it follows that
\begin{equation}\label{eqn-RHS4}\aligned
Y(\l_\a)\widetilde{J}Z  =  (J_\a B(Y, Z))^T +  A^{(J_\a Z)^N} (Y).
\endaligned
\end{equation}
Let $Z =e_j$ in (\ref{eqn-RHS4}), direct computation gives us
\begin{equation}\label{eqn-RHS5}\aligned
2Y(\l_\a) = &\sum_{j=1}^2 Y(\l_\a)\la \widetilde{J}e_j, \widetilde{J}e_j \ra = \sum_{j=1}^2 \la J_\a B(Y, e_j), \widetilde{J}e_j \ra + \sum_{j=1}^2 \la A^{(J_\a e_j)^N}(Y), \widetilde{J}e_j \ra \\
= & - \sum_{j=1}^2 \la  B(Y, e_j), J_\a\widetilde{J}e_j \ra +\sum_{j=1}^2 \la A^{(J_\a e_j)^N}(Y), \widetilde{J}e_j \ra
\endaligned
\end{equation}
Since $ \sum_{j=1}^2 \la  B(Y, e_j), J_\a\widetilde{J}e_j \ra$ is independent of the choice of local orthonormal frame fields of $\Sigma$ and $\la \widetilde{J}e_i, \widetilde{J}e_j \ra = \d_{ij}$, we obtain
\begin{equation}\label{eqn-RHS6}\aligned
 - \sum_{j=1}^2 \la  B(Y, e_j), J_\a\widetilde{J}e_j \ra = & \sum_{j=1}^2 \la B(Y, \widetilde{J}e_j), J_\a e_j \ra \\
 =&  \sum_{j=1}^2 \la B(Y, \widetilde{J}e_j), (J_\a e_j)^N \ra = \sum_{j=1}^2 \la A^{(J_\a e_j)^N}(Y), \widetilde{J}e_j \ra.
\endaligned
\end{equation}
From (\ref{eqn-RHS5}) and (\ref{eqn-RHS6}), we get
\[
Y(\l_\a) = \sum_{j=1}^2 \la A^{(J_\a e_j)^N}(Y), \widetilde{J}e_j \ra = \sum_{j=1}^2 \la A^{(J_\a e_j)^N}(\widetilde{J}e_j), Y\ra.
\]
It follows that
\begin{equation}\label{eqn-RHS7}\aligned
 \n \l_\a = \sum_{j=1}^2 e_j(\l_\a)e_j = \sum_{i=1}^2 A^{(J_\a e_i)^N}(\widetilde{J}e_i)
\endaligned
\end{equation}
and 
\begin{equation}\label{eqn-RHS8}\aligned
 \D \l_\a = &  \sum_{j=1}^2 e_j \left\la \sum_{i=1}^2 A^{(J_\a e_i)^N}(\widetilde{J}e_i), e_j \right\ra \\
 =&  \sum_{i, j=1}^2 e_j \left\la B(\widetilde{J}e_i, e_j), (J_\a e_i)^N \right\ra =  \sum_{i, j=1}^2 e_j \left\la B(\widetilde{J}e_i, e_j), J_\a e_i \right\ra \\
 =&  \sum_{i, j=1}^2 \left\la \overline{\n}_{e_j} B(\widetilde{J}e_i, e_j), J_\a e_i  \right\ra +  \sum_{i, j=1}^2 \left\la B(\widetilde{J}e_i, e_j), \overline{\n}_{e_j}(J_\a e_i)  \right\ra.
\endaligned 
\end{equation}
Using Weingarten formula in (\ref{eqn-RHS8}),
\begin{equation}\label{eqn-RHS9}\aligned
 \D \l_\a = &  \sum_{i, j=1}^2 \left\la -A^{B(\widetilde{J}e_i, e_j)}(e_j) + \n_{e_j}B(\widetilde{J}e_i, e_j), J_\a e_i   \right\ra +  \sum_{i, j=1}^2 \left\la B(\widetilde{J}e_i, e_j), J_\a \overline{\n}_{e_j} e_i  \right\ra \\
 =& - \sum_{i, j=1}^2\left\la B(e_j, (J_\a e_i)^T), B(\widetilde{J}e_i, e_j)  \right\ra + \sum_{i, j=1}^2 \left\la \n_{e_j}B(\widetilde{J}e_i, e_j), J_\a e_i  \right\ra \\
 &+ \sum_{i, j=1}^2 \left\la B(\widetilde{J}e_i, e_j), J_\a B(e_i, e_j)  \right\ra.
 \endaligned 
\end{equation}
Since 
\[
(J_\a e_i)^T = \l_\a \widetilde{J}e_i,
\]
we get
\begin{equation}\label{eqn-RHS10}\aligned
 B(e_j, (J_\a e_i)^T) = \l_\a B(e_j, \widetilde{J}e_i).
 \endaligned 
\end{equation}
By  Lemma \ref{lem:hyper-Lagrangian}, for any $\xi \in \Gamma(T^\bot \Sigma)$,
\[
( J_\a \xi )^N = \l_\a \widetilde{J}^{\bot} \xi.
\]
Thus
\begin{equation}\label{eqn-B}
(J_\a B(e_i, e_j))^N = \l_\a \widetilde{J}^{\bot} B(e_i, e_j).
\end{equation}
Then we have
\begin{equation}\label{eqn-RHS11}\aligned
\left\la B(\widetilde{J}e_i, e_j), J_\a B(e_i, e_j)  \right\ra = \left\la B(\widetilde{J}e_i, e_j), (J_\a B(e_i, e_j) )^N \right\ra = \l_\a \left\la B(\widetilde{J}e_i, e_j),  \widetilde{J}^{\bot} B(e_i, e_j)\right\ra.
 \endaligned 
\end{equation}
Substituting (\ref{eqn-RHS10}) and (\ref{eqn-RHS11}) into (\ref{eqn-RHS9}), we obtain
\begin{equation}\label{eqn-RHS12}\aligned
 \D \l_\a = &  - \l_\a \sum_{i, j=1}^2 \left\la B(\widetilde{J}e_i, e_j),  B(\widetilde{J}e_i, e_j) -  \widetilde{J}^{\bot} B(e_i, e_j) \right\ra + \sum_{i, j=1}^2 \left\la \n_{e_j}B(\widetilde{J}e_i, e_j), J_\a e_i  \right\ra.
 \endaligned 
\end{equation}
By using the Gauss formula again, 
\begin{equation}\label{eqn-RHS13}\aligned
 B(e_j, \widetilde{J}e_i) =& (\overline{\n}_{e_j}(\widetilde{J}e_i))^N = \left( \sum_{\b=1}^3 \overline{\n}_{e_j}(\l_\b J_\b e_i) \right)^N \\
 =& \left( \sum_{\b=1}^3 e_j(\l_\b)J_\b e_i + \sum_{\b=1}^3 \l_\b J_\b B(e_j, e_i) \right)^N \\
 =& \sum_{\b=1}^3 e_j(\l_\b) (J_\b e_i)^N + \sum_{\b=1}^3 \l_\b (J_\b B(e_j, e_i))^N.
 \endaligned 
\end{equation}
Since
\[
(J_\b e_i)^N = J_\b e_i - (J_\b e_i)^T = J_\b e_i - \l_\b \widetilde{J}e_i
\]
Substituting (\ref{eqn-B}) and the above equality into (\ref{eqn-RHS13}), we get
\begin{equation}\label{eqn-RHS14}\aligned
 B(e_j, \widetilde{J}e_i) =& \sum_{\b=1}^3e_j(\l_\b) J_\b e_i - \frac{1}{2}  e_j\left(\sum_{\b=1}^3 \l_\b^{2}\right)\widetilde{J}e_i + \sum_{\b=1}^3 \l_\b^{2}\widetilde{J}^\bot B(e_j, e_i) \\
 =&  \sum_{\b=1}^3e_j(\l_\b) J_\b e_i + \widetilde{J}^\bot B(e_j, e_i).
 \endaligned 
\end{equation}
By (\ref{eqn-J}), we derive
\begin{equation*}\label{eqn-RHS141}\aligned
\n_{e_j}B(\widetilde{J}e_i, e_j) = & (\n_{e_j}B)(\widetilde{J}e_i, e_j) + B\left( \n_{e_j}(\widetilde{J}e_i), e_j \right) + B\left( \widetilde{J}e_i, \n_{e_j}e_j  \right)   \\
=&   (\n_{e_j}B)(\widetilde{J}e_i, e_j)+  B\left( (\n_{e_j}\widetilde{J})e_i + \widetilde{J}\n_{e_j}e_i, e_j \right) + B\left( \widetilde{J}e_i, \n_{e_j}e_j  \right)  \\
=&  (\n_{e_j}B)(\widetilde{J}e_i, e_j). 
 \endaligned 
\end{equation*}
The Codazzi equation then implies
\begin{equation}\label{eqn-RHS15}\aligned
\sum_{j=1}^2\n_{e_j}B(\widetilde{J}e_i, e_j) =& \sum_{j=1}^2 (\n_{e_j}B)(\widetilde{J}e_i, e_j) =\sum_{j=1}^2 (\n_{\widetilde{J}e_i} B) (e_j, e_j) \\
=& \sum_{j=1}^2\n_{\widetilde{J}e_i} B (e_j, e_j) = \n_{\widetilde{J}e_i} H.
 \endaligned 
\end{equation}
 Substututing (\ref{eqn-RHS14}) and (\ref{eqn-RHS15}) into (\ref{eqn-RHS12}), we obtain
\begin{equation}\label{eqn-RHS16}\aligned
 \D \l_\a = &  - \l_\a \sum_{i, j=1}^2 \left\la B(\widetilde{J}e_i, e_j),   \sum_{\b=1}^3 e_j(\l_\b) J_\b e_i \right\ra + \sum_{i=1}^2 \left\la  \n_{\widetilde{J}e_i} H, J_\a e_i  \right\ra \\
 =& - \l_\a \sum_{i=1}^2 \sum_{\b=1}^3\left\la B(\widetilde{J}e_i, \n \l_\b),   J_\b e_i \right\ra + \sum_{i=1}^2 \left\la  \n_{\widetilde{J}e_i} H, J_\a e_i  \right\ra
 \endaligned 
\end{equation}
From (\ref{eqn-RHS7}) and (\ref{eqn-RHS16}), we have
\begin{equation*}\label{eqn-RHS166}\aligned
 \D \l_\a =& - \l_\a \sum_{i=1}^2 \sum_{\b=1}^3\left\la A^{(J_\b e_i)^N}(\widetilde{J}e_i), \n \l_\b \right\ra + \sum_{i=1}^2 \left\la  \n_{\widetilde{J}e_i} H, J_\a e_i  \right\ra \\
 =&  - \l_\a \sum_{\b=1}^3 |\n \l_\b|^2 + \sum_{i=1}^2 \left\la  \n_{\widetilde{J}e_i} H, J_\a e_i  \right\ra. 
 \endaligned 
\end{equation*}
Namely,
\begin{equation}\label{eqn-RHS17}
(\tau(J))^\a =  \sum_{i=1}^2 \left\la  \n_{\widetilde{J}e_i} H, J_\a e_i  \right\ra. 
\end{equation}
By (\ref{eqn-RHS17}), the translator equation (\ref{eqn-T111}) and the Weingarten formula, we get
\begin{equation}\label{eqn-T1}\aligned
(\tau(J))^\a = &\sum_{j=1}^2\left\la\nabla_{\widetilde{J}e_j}H, J_\alpha e_j \right\ra = - \sum_{j=1}^2\left\la \nabla_{\widetilde{J}e_j} V_{0}^{N}, J_\alpha e_j \right\ra \\
=  & -\sum_{j=1}^2\left\la \overline{\nabla}_{\widetilde{J}e_j} V_0^{N}, J_\alpha e_j \right\ra-  \sum_{j=1}^2\left\la A^{V_0^{N}}(\widetilde{J}e_j), J_\alpha e_j \right\ra \\
=& - \sum_{j=1}^2\left\la \overline{\nabla}_{\widetilde{J}e_j} V_0, J_\alpha e_j \right\ra + \sum_{j=1}^2 \left\la \overline{\nabla}_{\widetilde{J}e_j} V_0^{T}, J_\alpha e_j \right\ra \\
&-  \sum_{j=1}^2 \left\la B(\widetilde{J}e_j, (J_\a e_j)^T), V_0^{N} \right\ra\\
=&  \sum_{j=1}^2 \left\la \overline{\nabla}_{\widetilde{J}e_j} V_0^{T}, J_\alpha e_j \right\ra -  \sum_{j=1}^2 \left\la B(\widetilde{J}e_j, (J_\a e_j)^T), V_0^{N} \right\ra.
\endaligned
\end{equation}
Since the translating soliton surface $\Sigma$ is hyper-Lagrangian, by Lemma \ref{lem:hyper-Lagrangian} ,we have
\begin{equation}\label{eqn-J12}
(J_\a e_j)^T = \l_\a \widetilde{J}e_j.
\end{equation}
Thus we obtain
\begin{equation}\label{eqn-T2}\aligned
 \sum_{j=1}^2\left\la B(\widetilde{J}e_j, (J_\a e_j)^T), V_0^{N} \right\ra = &  \sum_{j=1}^2 \left\la B(\widetilde{J}e_j, \l_\a \tilde{J}e_j), V_0^{N} \right\ra \\
 = & \l_\a \left\la H, V_0^{N} \right\ra = \l_\a \left\la H, V_0 \right\ra.
\endaligned
\end{equation}
On the other hand, by the Gauss formula, we derive
\begin{equation}\label{eqn-T3}\aligned
  & \sum_{j=1}^2\left\la \overline{\nabla}_{\widetilde{J}e_j} V_0^{T}, J_\alpha e_j \right\ra =   \sum_{j=1}^2 \left\la \overline{\n}_{\widetilde{J}e_j}\left(\sum_{k=1}^2\left\la V_0, e_k \right\ra e_k\right), J_\a e_j \right\ra \\
  =& \sum_{j, k=1}^2 \left\la V_0, \overline{\n}_{\widetilde{J}e_j}e_k \right\ra \left\la e_k, J_\a e_j \right\ra  + \sum_{j, k=1}^2 \left\la V_0, e_k \right\ra \left\la \overline{\n}_{\widetilde{J}e_j} e_k, J_\a e_j \right\ra \\
 =& \sum_{j, k=1}^2\left\la V_0, B(\widetilde{J}e_j, e_k)\right\ra \left\la e_k, J_\alpha e_j \right\ra+\sum_{j, k=1}^2 \left\la V_0, e_k\right\ra \left\la B(\widetilde{J}e_j, e_k), J_\alpha e_j \right\ra  \\
=& \sum_{j=1}^2 \left\la V_0, B(\widetilde{J}e_j, (J_\a e_j)^T) \right\ra + \sum_{j=1}^2 \left\la B(\widetilde{J}e_j, V_0^{T}), J_\a e_j \right\ra. 
\endaligned
\end{equation}
Using the formula (\ref{eqn-J12}) again, we then get
\begin{equation}\label{eqn-T33}\aligned
   \sum_{j=1}^2\left\la \overline{\nabla}_{\widetilde{J}e_j} V_0^{T}, J_\alpha e_j \right\ra =& \sum_{j=1}^2 \left\la V_0,  B(\widetilde{J}e_j, \l_\a \widetilde{J}e_j) \right\ra + \sum_{j=1}^2 \left\la B(\widetilde{J}e_j, V_0^{T}), (J_\a e_j)^N \right\ra \\
=&   \lambda_\alpha \left\la V_0, H \right\ra +\sum_{j=1}^2 \left\la A^{(J_{\alpha}e_j)^N}(\widetilde{J}e_j), V_{0}^{T}\right\ra.
\endaligned
\end{equation}
Substituting (\ref{eqn-T2}) and (\ref{eqn-T33}) into (\ref{eqn-T1}), and using (\ref{eqn-RHS7}), we have
\begin{equation*}\aligned
    (\tau(J))^\alpha =& \sum_{j=1}^2\left\la \nabla_{\widetilde{J}e_j}H, J_\alpha e_j \right\ra = \sum_{j=1}^2\left\la A^{(J_{\alpha}e_j)^N}(\widetilde{J}e_j), V_{0}^{T}\right\ra \\
    =& \left\la \nabla \lambda_\alpha, V_0^{T}\right\ra = (dJ(V_0^{T}))^\a, \quad \quad \a= 1, 2, 3.
 \endaligned
\end{equation*}
Namely
\begin{equation*}
\tau\left(J\right)= dJ\left(V_0^{T}\right).
\end{equation*} 
\end{proof}

\vskip12pt

\section{A rigidity result}

Let $\Sigma^2$ be a 2-dimensional translating soliton in $\mathbb{R}^4$. 
Let $ \kappa^\bot:= \la R(e_1, e_2) \nu_1, \nu_2 \ra $ be the normal curvature, 
where $\{e_1, e_2\}$ is a local orthonormal frame field on $\Sigma$, $\{ \nu_1, \nu_2 \}$ is a local orthonormal frame normal field along $\Sigma$, and the orientation of $\{ e_1, e_2, \nu_1, \nu_2 \}$ is coincided with the one of $\{e_1, \widetilde{J}e_1, \widetilde{J}_2 e_1, \widetilde{J}_3 e_1\}$. Here  
$\widetilde{J}_2, \widetilde{J}_3 $ is the same as the ones in the proof of Lemma \ref{lem:hyper-Lagrangian}.

Let $V:= -V_0^{T}$ and $\Delta_V:=\Delta + \la V, \n \cdot \ra$.

By using Theorem \ref{thm:translator} and gradient estimates, we obtain a rigidity result of translating solitons.

\begin{thm}\label{thm:soliton}

Let $X: \Sigma^2 \to \mathbb{R}^4$ be a 2-dimensional complete translating soliton with nonpositive normal curvature. Assume that the image of the complex phase map is contained in a regular ball in $\mathbb{S}^2$, i.e., a geodesic ball $B_R(q)$ disjoint from the cut locus of $q$ and $R < \frac{\pi}{2}$, then $\Sigma$ has to be an affine plane.

\end{thm}

\begin{proof}

Since we can view $\Sigma$ as a hyper-Lagrangian submanifold in $\mathbb{R}^4$ with respect to some almost complex structure $\widetilde J$,   Let $\{ e_1, e_2=\widetilde{J}e_1 \}$ be a local orthonormal frame field on $\Sigma$  such that $\nabla e_i =0$ at the considered point. Denote $\nu_1=\widetilde{J}_2 e_1, \nu_2 =\widetilde J^{\bot}\nu_1= \widetilde{J_1}\widetilde J_2e_1 =\widetilde{J}_3 e_1$, where $\widetilde{J}_\b (\b = 1, 2, 3)$ is the same as the ones in the proof of Lemma \ref{lem:hyper-Lagrangian}. Then $\{\nu_1, \nu_2\}$  is a local orthonormal frame  normal field along $\Sigma$.

 From the translating soliton equation (\ref{eqn-T111}), we derive
\begin{equation*}
\n_{e_j} H = - \left( \overline{\n}_{e_j}(V_0-\la V_0, e_k \ra e_k) \right)^N = \la V_0, e_k \ra B_{jk}
\end{equation*}
and
\begin{equation*}
\n_{e_i}\n_{e_j} H = \la V_0, e_k \ra \n_{e_i} B_{jk} -\la H, B_{ik} \ra B_{jk},
\end{equation*}
where $B_{jk}=B(e_j, e_k)$.
Hence using the Codazzi equation, we obtain 
 \begin{equation*}\aligned
\D_V |H|^2 = & \D |H|^2 +\la V, \n |H|^2 \ra  \\
=& 2\la \n_{e_i}\n_{e_i}H, H \ra +2|\n H|^2+ \la V, \n |H|^2\ra \\
=& -2 \la H, B_{e_i e_k} \ra^2 +2\la \n_{V_0^{T}} H, H \ra +2|\n H|^2 + \la V, \n |H|^2\ra \\
=& -2 \la H, B_{e_i e_k} \ra^2 - \n_V |H|^2 +2|\n H|^2 + \la V, \n |H|^2\ra \\
=&  -2 \la H, B_{e_i e_k} \ra^2  + 2|\n H|^2.
\endaligned
\end{equation*}
It follows that 
 \begin{equation}\label{eqn-MC}\aligned
\D_V |H|^2 \geq   2|\n H|^2-2|B|^2|H|^2.
\endaligned
\end{equation}
By (\ref{eqn-RHS14}), we get
 \begin{equation*}\aligned
B(e_j, \widetilde{J}e_k) - \widetilde{J}^\bot B(e_j, e_k)=\sum_{\b=1}^3 e_j(\l_\b)J_\b e_k.
\endaligned
\end{equation*}
The above equality implies that
 \begin{equation}\label{eqn-JB}\aligned
|dJ|^2 = &\sum_{j=1}^2 |dJ(e_j)|^2 =\frac{1}{2}\sum_{j, k=1}^2 \left|dJ(e_j)e_k\right|^2 
= \frac{1}{2}\sum_{j, k=1}^2 \left| B(e_j, \widetilde{J}e_k) - \widetilde{J}^\bot B(e_j, e_k) \right|^2 \\
=& \frac{1}{2}\sum_{j, k=1}^2  \left| B(e_j, \widetilde{J}e_k)\right|^2 +  \frac{1}{2}\sum_{j, k=1}^2 \left|\widetilde{J}^\bot B(e_j, e_k) \right|^2 -\sum_{j, k=1}^2\la B(e_j, \widetilde{J}e_k),  \widetilde{J}^\bot B(e_j, e_k) \ra \\
=& |B|^2 -\sum_{j, k=1}^2\la B(e_j, \widetilde{J}e_k),  \widetilde{J}^\bot B(e_j, e_k) \ra 
\endaligned
\end{equation}
Applying the Ricci equation, we obtain 
\begin{equation}\label{eqn-BK}\aligned
& \sum_{j, k=1}^2\la B(e_j, \widetilde{J}e_k),  \widetilde{J}^\bot B(e_j, e_k) \ra =  \sum_{j, k, s=1}^2 \la B(e_j, \widetilde{J}e_k), \nu_s \ra \la \widetilde{J}^\bot B(e_j, e_k), \nu_s \ra \\
=& -\sum_{j, k=1}^2 \la B(e_j, \widetilde{J}e_k), \nu_1 \ra \la B(e_j, e_k), \nu_2 \ra + \sum_{j, k=1}^2 \la B(e_j, \widetilde{J}e_k), \nu_2 \ra \la B(e_j, e_k), \nu_1 \ra \\
=&   \sum_{k=1}^2 \la R(e_k, \widetilde{J}e_k) \nu_1, \nu_2 \ra= 2 \la R(e_1, e_2) \nu_1, \nu_2 \ra =2\kappa^\bot.
\endaligned
\end{equation}
Substituting (\ref{eqn-BK}) into  (\ref{eqn-JB}), we have
\begin{equation}\label{eqn-JBK}\aligned
|dJ|^2 = |B|^2 - 2\kappa^\bot.
\endaligned
\end{equation} 

 Let $\rho $ be the distance function on $\mathbb{S}^2$, and $h$ the Riemannian metric of $\mathbb{S}^2$. Define $\psi=1-\cos\rho$,  then ${\rm Hess}(\psi) =( \cos\rho)h$. 

For any $X=(x_1, ..., x_{4}) \in \mathbb{R}^{4}$, let $r=|X|$, then we have 
\begin{equation}\label{eqn-ER}\aligned
\nabla r^2 =&  2X^{T},  \quad |\nabla r| \leq 1 \\
\Delta r^2 = & 4 + 2\la H,  X \ra \leq 4+2r.
\endaligned
\end{equation}

Since $J(\Sigma) \subset B_R(q)\subset \mathbb{S}^2$, note that $R<\frac{\pi}{2}$, so we can choose a constant $b$, such that $\psi(R)< b < 1.$ Let $B_a(o)$ be the closed ball centered at the origin $o$ with radius $a $ in $\mathbb{R}^{4}$ and $D_a(o):= \Sigma\cap B_a(o)$. Define $f: D_a(o) \rightarrow \mathbb{R}$ by
\[
f=\frac{(a^2-r^2)^2|H|^2}{(b-\psi \circ J)^2}.
\]

Since $\left. f \right|_{\partial D_{a}(o)}=0$, $f$ achieves an absolute maximum in the interior of
$D_{a}(o)$, say $f\leq f(q)$, for some $q$ inside $D_{a}(o)$. By using the technique of support
function we may assume that $f$ is smooth near $q$. We may also assume $|H|(q) \neq 0$.
Then from
\begin{equation*}
\label{V3.3add}
          \begin{array}{rcl}
           \ds\vs \nabla f(q)&=&0, \\
       \ds \D_V f(q) & \leq & 0,
        \end{array}
\end{equation*}
we obtain the following at the point $q$:
\begin{equation}\label{3.3}
-\frac{2\nabla r^{2}}{a^{2}- r^{2}} + \frac{\nabla |H|^2}{|H|^2}
 + \frac{2\nabla (\psi\circ J)}{b - \psi\circ J} = 0,
\end{equation}
\begin{equation}\label{3.4}
- \frac{2 \D_V r^{2}}{a^{2}-r^{2}} - \frac{2 \left| \nabla r^{2}
\right|^{2}}{\left( a^{2} - r^{2} \right)^{2}} + \frac{\D_V
|H|^2}{|H|^2} - \frac{\left|\nabla |H|^2\right|^{2}}{|H|^4} +
\frac{2 \D_V (\psi \circ J)}{b-\psi \circ J} + \frac{2\left|
\nabla (\psi \circ J) \right|^{2}}{\left( b-\psi \circ J
\right)^{2}} \leq 0.
\end{equation}
Direct computation gives us
\begin{equation}\label{eqn-GMC}\aligned
|\n |H|^2|^2 = |2\la \n H, H \ra|^2 \leq 4|\n H|^2 |H|^2,
\endaligned
\end{equation}
\begin{equation}\label{eqn-PJ}\aligned
|\n(\psi\circ J)|^2 \leq |d\psi|^2|dJ|^2 \leq |dJ|^2.
\endaligned
\end{equation}
It follows from (\ref{eqn-MC}) and (\ref{eqn-GMC}) that
\begin{equation}\label{eqn-MB}\aligned
\frac{\D_V|H|^2}{|H|^2} \geq \frac{|\n |H|^2|^2}{2|H|^4} - 2|B|^2.
\endaligned
\end{equation}
From (\ref{3.3}), we obtain 
\begin{equation}\label{eqn-MCRJ}\aligned
\frac{|\n |H|^2|^2}{|H|^4} \leq \frac{4|\n r^2|^2}{(a^2-r^2)^2}+ \frac{8|\n r^2||\n(\psi\circ J)|}{(a^2-r^2)(b-\psi\circ J)} + \frac{4|\n(\psi\circ J)|^2}{(b-\psi\circ J)^2}
\endaligned
\end{equation}
By Theorem \ref{thm:translator}, $\tau_V(J) = 0$. Thus we get
\begin{equation}\label{eqn-PJ2}\aligned
\D_V(\psi\circ J) = \sum_{j=1}^2 {\rm Hess}(\psi)(dJ(e_j), dJ(e_j)) + d\psi(\tau_V(J)) = \cos\rho |dJ|^2.
\endaligned
\end{equation}
Substituting (\ref{eqn-JBK}), (\ref{eqn-ER}), (\ref{eqn-PJ}), (\ref{eqn-MB}), (\ref{eqn-MCRJ}), (\ref{eqn-PJ2}) into (\ref{3.4}), we have 
\begin{equation*}\label{eqn-GJ}\aligned
\left( \frac{\cos \rho}{b-\psi\circ J} -1 \right) |dJ|^2 - \frac{4r}{(a^2-r^2)(b-\psi\circ J)}|dJ| -\frac{4(1+r)}{a^2-r^2} -\frac{8r^2}{(a^2-r^2)^2} -2\kappa^\bot \leq 0.
\endaligned
\end{equation*}
By the assumption that $\k^\bot \leq 0$, we derive
\begin{equation*}\label{eqn-GJ}\aligned
\left( \frac{\cos \rho}{b-\psi\circ J} -1 \right) |dJ|^2 - \frac{4r}{(a^2-r^2)(b-\psi\circ J)}|dJ| -\frac{4(1+r)}{a^2-r^2} -\frac{8r^2}{(a^2-r^2)^2}  \leq 0.
\endaligned
\end{equation*}
Note an elementary fact that if $ax^{2}-bx-c \leq 0$ with $a, b, c
$ all positive, then
\begin{equation*}
x \leq \max \{ 2b/a, 2\sqrt{c/a} \}.
\end{equation*}
It is easy to see that there is a constant $C>0$ such that
$\frac{ \cos\rho}{b-\psi\circ J}  -1 >C.$  Therefore, at the
point $q$,
\begin{equation}\aligned\label{3.16}
|dJ|^2   \leq & \max \left\{  \frac{64 r^{2}}
{C^{2}(a^{2}-r^{2})^{2}(b-\psi\circ J)^{2}},  \frac{16(1+r)}{C(a^{2}-r^{2})} 
 +  \frac{32r^{2}}{C\left( a^{2}-r^{2} \right)^{2}}
 \right\}.
\endaligned
\end{equation}
By (\ref{eqn-JBK}) and $\k^\bot \leq 0$, we get
\[
|dJ|^2 \geq |B|^2 \geq \frac{|H|^2}{2}.
\]
Thus we obtain,
 at the point $q$,
\begin{equation}\aligned\label{3.16}
|H|^2  \leq  2\max \left\{  \frac{64 r^{2}}
{C^{2}(a^{2}-r^{2})^{2}(b-\psi\circ J)^{2}},  \frac{16(1+r)}{C(a^{2}-r^{2})} 
 +  \frac{32r^{2}}{C\left( a^{2}-r^{2} \right)^{2}} 
 \right\}.
\endaligned
\end{equation}
and 
\begin{equation*}\aligned\label{3.16}
f(q)  \leq  2\max \left\{  \frac{64 a^{2}}
{C^{2}\left(b-\psi(\widetilde{R})\right)^{4}},  \frac{16(1+a)a^2}{C\left(b-\psi(\widetilde{R})\right)^{2}} 
 +  \frac{32a^{2}}{C\left( b-\psi(\widetilde{R}) \right)^{2}} 
 \right\}.
\endaligned
\end{equation*}
Then for any point $x \in D_{a/2}(o)$, we have
\begin{equation}\aligned\label{3.16'}
|H|^2(x)  \leq & \frac{(b-\psi\circ J)^2}{(a^2-r^2)^2}f(q) \\
\leq & \frac{32b^2}{9a^4} \max \left\{  \frac{64 a^{2}}
{C^{2}\left( b-\psi(\widetilde{R}) \right)^{4}},  \frac{16(1+a)a^2}{C\left( b-\psi(\widetilde{R}) \right)^{2}} 
 +  \frac{32a^{2}}{C\left( b-\psi(\widetilde{R}) \right)^{2}} 
 \right\}.
\endaligned
\end{equation}
Hence we may fix $x$ and let $a \rightarrow \infty$ in (\ref{3.16'}), we then derive that $H\equiv 0.$ 
Then by Proposition 3.2 in \cite{HanLi09}, $B \equiv 0$. Namely, $\Sigma $ is an affine plane.
\end{proof}

\begin{rem}
\begin{itemize}
    \item[(1)] Let $\alpha$ be the K\"ahler angle of the translator, Theorem \ref{thm:soliton} implies that if $\cos\alpha$ has a positive lower bound, then any complete symplectic translating soliton with nonpositive normal curvature has to be an affine plane. Han-Sun  \cite{HanSun10} showed that such a rigidity result holds under an additional bounded second fundamental form assumption.  
      \item[(2)] The authors \cite{QS19} proved that if the image of the complex phase map is contained in a regular ball in $\mathbb{S}^2$, then any 2-dimensional complete translating soliton with flat normal bundle must be flat. Thus this result is improved by the above Theorem \ref{thm:soliton}.
    \item[(3)] The restriction of the image under the complex phase map in Theorem \ref{thm:soliton} is necessary. For example, the ``grim reaper" $(x, y, -\ln\cos x,0), |x|<\pi/2, y\in\mathbb{R}$ is a translating soliton to the symplectic MCF which translates in the direction of the constant vector $(0,0,1,0)$, and $J=(\cos x,0,-\sin x), |x|<\pi/2$ can not contained in any regular ball of $\mathbb{S}^2$. One can check that $|B|^2=|H|^2=|J|^2=\cos^2x$ and the normal curvature is zero. 
\end{itemize}

\end{rem}

Due to the fact that, in a hyperk\"ahler 4-manifold, a surface being symplectic is equivalent to the condition that the image under the complex phase map is contained in an open hemisphere while a surface being Lagrangian is equivalent to the condition that the image under the complex phase map is contained in a great circle. 
Theorem \ref{thm:soliton} implies the following:

\begin{cor}\label{cor-1}

Let $X: \Sigma^2 \to \mathbb{R}^4$ be a complete Lagrangian translating soliton with nonpositive normal curvature. If the cosine of the Lagrangian angle has a positive lower bound, then $\Sigma$ has to be an affine plane.

\end{cor}

\begin{rem}
Using the Gauss equation, we get
\[
|B|^2 = |H|^2 - 2\k,
\]
where $\k$ is the sectional curvature of $\Sigma$. Then by the above equality
 and (\ref{eqn-JBK}), we can conclude that 
\[
|dJ|^2 =|H|^2 - 2(\kappa + \kappa^\bot).
\]
For Lagrangian surfaces, 
the complex phase map $J: \Sigma \to \mathbb{S}^2$ can be represented by $(\cos \theta, \sin \theta, 0)$, thus direct computation gives
\begin{equation*}\aligned\label{eqn-L}
|dJ|^2 = &\sum_{j} |dJ(e_j)|^2 = \sum_{j} |e_j(\cos \theta)J_1 + e_j(\sin\theta)J_2|^2 \\
= &\sum_j(-\sin\theta e_j(\theta))^2 + ( \cos\theta e_j(\theta) )^2 =|\n \theta|^2 \\
=& |\widetilde{J}\n \theta|^2 =|H|^2.
\endaligned
\end{equation*}
The above two equalities imply that for Lagrangian surfaces, we have $\kappa + \kappa^\bot =0$. This shows that for Lagrangian surfaces, $\kappa^\bot \leq 0$ if and only if $\kappa \geq 0$. Therefore, Corollary \ref{cor-1} is equivalent to the Main Theorem 2 in \cite{HanSun10}.
\end{rem}

\vskip 10pt

\noindent{\bf Acknowledgements} The author would like to thank Dr. L. Sun for helpful discussions.


\vskip24pt


\begin{thebibliography}{1}


\bibitem{AV95} Angenent, S. B., Velazquez, J. J. L., Asymptotic shape of cusp singularities in curve shortening, {\it Duke Math. J.} {\bf 77}(1995), no. 1, 71--110.


\bibitem{AV97} Angenent, S. B., Velazquez, J. J. L., Degenerate neckpinches in mean curvature flow, {\it J. Reine Angew Math.} {\bf 482} (1997), 15--66. 

  
\bibitem{BS} Bao, C., Shi, Y., Gauss map of translating solitons of mean curvature flow, {\it Proc. Amer. Math. Soc.} {\bf 142} (2014), 4333--4339.


\bibitem{CheJosQiu12}
Chen, Q., Jost, J., Qiu, H., Existence and {L}iouville theorems for
  {$V$}-harmonic maps from complete manifolds, {\it Ann. Global Anal. Geom.} {\bf 42} (2012), no. 4,
  565--584.



\bibitem{CheJosWan15}
Chen, Q., Jost, J., Wang, G., A maximum principle for generalizations of
  harmonic maps in Hermitian, affine, Weyl, and Finsler geometry, {\it J.
  Geom. Anal.} {\bf 25} (2015), no. 4, 2407--2426.


\bibitem{CheLi01}
Chen, J., Li, J., Mean curvature flow of surface in {$4$}-manifolds. {\it Adv.
  Math.} {\bf 163} (2001), no. 2, 287--309.

\bibitem{CheQiu16}
Chen, Q., Qiu, H., Rigidity of self-shrinkers and translating solitons of
  mean curvature flows, {\it Adv. Math.} {\bf 294} (2016), 517--531.



\bibitem{ColMin12}
Colding, T. H., Minicozzi, II, W. P., Generic mean curvature flow {I}:
  generic singularities, {\it Ann. of Math.}  {\bf (2) 175} (2012), no. 2, 755--833.


\bibitem{CSS} Clutterbuck, J., Schnfirer, O. C., Schulze, F., Stability of translating solutions to mean curvature flow, {\it Calc. Var. Partial Differential
  Equations} {\bf 29} (2007), 281--293.


\bibitem{Hal} Halldorsson, H. P., Helicoidal surfaces rotating/translating under the mean curvature flow, {\it Geom. Dedicata} {\bf 162} (2013), 45--65.

\bibitem{Ham95} 
Hamilton, R. S., Harnack estimate for the mean curvature flow, {\it J. Differential Geom.} {\bf 41} (1995), 215--226.

\bibitem{HanLi09}
Han, X., Li, J., Translating solitons to symplectic and Lagrangian mean
  curvature flows, {\it Internat. J. Math.} {\bf 20} (2009), no. 4, 443--458.


\bibitem{HanSun10}
Han, X., Sun, J., Translating solitons to symplectic mean curvature
  flows, {\it Ann. Global Anal. Geom.} {\bf 38} (2010), no. 2, 161--169.


\bibitem{HuiSin99}
Huisken, G., Sinestrari, C., Convexity estimates for mean curvature flow
  and singularities of mean convex surfaces, {\it Acta Math.} {\bf 183} (1999), no. 1, 45--70.


\bibitem{HuiSin2-99}
Huisken, G., Sinestrari, C., Mean curvature flow singularities for mean convex surfaces, {\it Calc. Var. Partial Differential Equations} {\bf 8} (1999), 1--14.

\bibitem{Jian} Jian, H. Y., Translating solitons of mean curvature flow of noncompact spacelike hypersurfaces in Minkowski space, {\it J. Differential Equations} {\bf 220} (2006), 147--162.


\bibitem{Kun} Kunikawa, K., Bernstein-type theorem of translating solitons in arbitrary codimension with flat normal bundle. {\it Calc. Var. Partial Differential Equations} {\bf 54} (2015), no. 2, 1331--1344.

\bibitem{LeuWan07}
Leung, N.~C., Wan, T. Y.~H., Hyper-Lagrangian submanifolds of
  hyperk\"{a}hler manifolds and mean curvature flow, {\it J. Geom. Anal.} {\bf 17} (2007), no. 2,
  343--364.


\bibitem{Mos} Moser, J., On Harnack’s theorem for elliptic differential equations, {\it Comm.Pure Appl. Math.} {\bf 14}(1961), 577--591.

\bibitem{MSS} Martín, H., Savas-Halilaj, A. S., Smoczyk, K., On the topology of translating solitons of the mean curvature flow, {\it Calc. Var. Partial Differential Equations} {\bf 54}(2015), 2853--2882.


\bibitem{Ngu09} Nguyen, X. H., Translating tridents, {\it Commun. Partial Differential Equations} {\bf 34} (2009), 257--280.

\bibitem{Ngu13} Nguyen, X. H., Complete embedded self-translating surfaces under mean curvature flow, {\it J. Geom. Anal.} {\bf 23} (2013), 1379--1426.

\bibitem{NT} Neves, A., Tian, G., Translating solutions to Lagrangian mean curvature flow, {\it Trans. Amer. Math. Soc.} {\bf 365} (2013), no. 11, 5655–5680.

\bibitem{Qiu} Qiu, H. B., A Bernstein type result of translating solitons, {\it arXiv:2204.12744}, 2022.

\bibitem{QS} Qiu, H. B., Sun, L. L., Rigidity of self-shrinking surfaces and its applications, {\it Preprint} 2021.

\bibitem{QS19} Qiu, H. B., Sun, L. L., Mean curvature flow of surfaces in a hyperk\"ahler 4-manifold, {\it arXiv:1902.00645v1}, 2019.


\bibitem{Wan01}
Wang, M.-T., Mean curvature flow of surfaces in Einstein
  four-manifolds, {\it J. Differential Geom.} {\bf 57} (2001), no. 2, 301--338.


\bibitem{Wang11} Wang, X.-J., Convex solutions to the mean curvature flow, {\it Ann. of Math.} {\bf 173} (2011), 1185--1239.

\bibitem{Whi00} White, B., The size of the singular sets in mean curvature flow of mean convex sets, {\it J. Amer. Math. Soc.}
{\bf 13}(2000), 665--695.

\bibitem{Whi03} White, B., The nature of singularities in mean curvature flow of mean convex sets, {\it J. Amer. Math. Soc.} {\bf 16} (2003), 
123--138.


\bibitem{Xin}
Xin, Y. L.,  Translating solitons of the mean curvature flow, {\it Calc. Var. Partial Differential Equations
  } {\bf 54} (2015), 1995--2016.



\end{thebibliography}
\end{document}